\documentclass[11pt]{amsart}
\usepackage{latexsym}
\usepackage{amssymb}
\usepackage{amsfonts}
\usepackage{graphicx}
\usepackage{epsf}
\usepackage[all]{xypic}
\usepackage{delarray}
\usepackage{setspace}
\newtheorem{theorem}{Theorem}[section]
\newtheorem{corollary}[theorem]{Corollary}
\newtheorem{lemma}[theorem]{Lemma}
\newtheorem{proposition}[theorem]{Proposition}
\newtheorem{definition}[theorem]{Definition}

\newtheorem{remark}[theorem]{Remark}
\newcommand{\Hom}{{\rm Hom}}

\definemorphism{dato}\dashed \tip \notip
\definemorphism{Dato}\Dashed \tip \notip

\newcommand{\Mm}{\mathcal{M}}

\newcommand{\Tt}{\mathcal{T}}
\newcommand{\Ss}{\mathcal{S}}
\newcommand{\Uu}{\mathcal{U}}

\def\NN{{\mathbb N}}

\begin{document}
\sloppy

\title[On two conjectures of Faith]{On two conjectures of Faith}


\begin{abstract}
We prove that a profinite algebra whose left (right) cyclic modules are torsionless is finite dimensional and QF. We give a relative version of the notion of left (right) PF ring for pseudocompact algebras and prove it is left-right symmetric and dual to the notion of quasi-co-Frobenius coalgebras. We also prove two ring theoretic conjectures of Faith, in the setting (and supplementary hypothesis) of profinite algebras: any profinite semiartinian selfinjective algebra is finite dimensional and QF, and any FGF profinite algebra is finite dimensional QF.
\end{abstract}

\author{Mariana Haim, M.C.Iovanov, Blas Torrecillas}


\maketitle

\section{Introduction}

A ring is called quasi Frobenius (qF) if it is right or left self-injective and right
or left artinian (all four combinations being equivalent). The study of these rings
grew out of the theory of representations of a finite group: if $G$ is finite, the group
algebra $\Bbbk G$ is quasi Frobenius (in fact, it is even a Frobenius algebra). Quasi-Frobenius rings capture the module theoretic properties of group algebras, or more generally, of Frobenius algebras.

A theorem of Faith and Walker \cite{FW} states that a ring R is quasi Frobenius if
and only if any left (or right) R-module embeds in a free left (or right) R-module.
More generally, a ring R is said to be left (right) FGF if this condition holds
for finitely generated left (right) R-modules. The FGF conjecture states that
every left FGF ring is quasi-Frobenius. It has been proved in several contexts,
in particular when the ring verifies some finiteness condition (for left or right
noetherian rings, for semiregular rings whose Jacobson radical is T-nilpotent)
and also when there is some duality between left and right R-modules (for left
and right FGF rings, in particular for commutative rings). In his survey \cite{FH}, Faith
gives a proof for completely FGF rings, i.e. in the case that every factor ring of
R is left FGF.

It is to expect that pseudocompact and profinite algebras are a good context to allow the FGF conjecture to hold, since they verify some sort of finiteness condition (the ``locally finite'' property of coalgebras and comodules). Pseudocompact algebras are topological algebras with a basis of neighborhoods of $0$ consisting of cofinite ideals, and which are Hausdorff complete; they are precisely the algebras arrising as duals of coalgebras, and they are always profinite. Also, any profinite algebra (i.e. prolimit of finite dimensional algebras) admits at least one such pseudocompact topology, i.e. is the dual of a coalgebra. Indeed, in this paper we prove that the FGF conjecture holds for pseudocompact algebras. Actually, we prove something more general and quite less expectable: every pseudocompact algebra whose left (or right) cyclic modules are torsionless is finite dimensional QF.

Another important open question is what has come to be called Faith's conjecture: ``Any left self-injective semiprimary ring is QF'' (see \cite{FH} and \cite{NY} for more information about this conjecture). We show that if $A$ is left semiartinian and left self-injective then $A$ is a finite dimensional QF-algebra, which is a more general statement than the above conjecture.

A ring $R$ such that every faithful left $R$-module generates the category $R$-Mod of left $R$-modules is called left pseudo-Frobenius (PF). In \cite{A}, Azumaya introduced this type of rings as a generalization of QF rings. We will see (Theorem 3.8) that for profinite algebras the notion is symmetric and also implies that the algebra is finite dimensional QF.

We also propose suitable analogues for the notions of PF in the setting of pseudocompact algebras. A ring $R$ is left PF if and only if it is left selfinjective and any left simple module embeds in $R$. For a pseudocompact algebra, there is a coalgebra $C$ such that $A=C^*$ and a class of rational $A$-modules. It is then natural to weaken the PF property to left ``Rat-PF'' by asking that $A$ is left self-injective and left simple {\it rational} modules embed in $A$. We show that this is also left-right symmetric and is precisely characterizing another important analogue notion introduced in the theory of coalgebras, that is, those of quasi-co-Frobenius coalgebras, or shortly QcF coalgebras. Recall that a coalgebra $C$ is said to be left QcF if and only if $C$ embeds in a product (equivalently, coproduct) power of $C^*$; this is also equivalent to saying that $C$ is projective as a left $C^*$-module (equivalently, right $C$-comodule). We show that $A=C^*$ is left (or right) Rat-PF if and only if $C$ is left and right Quasi-co-Frobenius. This connects the coalgebra notions with the module properties of the dual convolution algebra.

\section{Preliminaries}\label{pre}
In this section we fix the basic notations and recall a few definitions and basic concepts in coalgebra theory. We refer to \cite{DNR} and \cite{Sweedler} for more details.\\
The coalgebras we will work with will be $\Bbbk$-coalgebras, where $\Bbbk$ is a field. For coalgebras and comodules we use Sweedler's notation. We denote by $\ ^C{\mathcal M}$ and ${\mathcal M}^C$ the categories of left and right comodules over $C$ respectively. It is well known that both are Grothendieck categories and therefore they have enough injective objects. Similarly, if $A$ is an algebra over $\Bbbk$, we denote by $\ _A{\mathcal M}$ and ${\mathcal M}_A$ the categories of left and right modules over $A$ respectively.

Let $V$ be a vector space. If $X$ is a subset of $V$ and $Y$ is a subset of $V^*$, then
define $X^\perp = \{f\in V^* \mid f(x) = 0 \ \forall x\in X\}\subseteq  V^*$ and $Y ^\perp = \{ v \in V \mid f(v) =0 \ \forall f\in Y\}\subseteq V
$.\\
The operator $\perp \perp$ is a closure operator on subspaces of $V^*$. The induced topology is called the finite topology of $V^*$. It is well known that for every subspace $S \subseteq V$, we have $ S^{\perp \perp} = S$ and that if $T$ is
a finite dimensional subspace of $V^*$ then $T^{\perp \perp} = T$ (i.e. every finite dimensional subset of $V^*$ is closed under the finite topology).

Let $C$ be a coalgebra. Consider $C$ as a right $C$-comodule. We recall that the coradical of $C$, denoted by $C_0$, is the sum of all simple right coideals in $C$.\\
We can define by recursion an ascending chain
$$0 = C_0 \subseteq C_1 \subseteq C_2 \subseteq \cdots \subseteq C_n \subseteq \cdots $$
of right subcomodules of $C$ as follows. Let $C_0 = soc(C)$ be the coradical of $C$ and for any $n \in \NN$ define $C_{n+1}$ such that $soc\left (\frac{C}{C_n}\right ) = \frac{C_{n+1}}{C_n}$. It can be proved that if we make the construction by regarding $C$ as a left comodule, we obtain the same chain. Thus, $C_n$ is a left and right subcomodule of $C$; more precisely, $C_n$ is a subcoalgebra of $C$. The above defined chain is called the \emph{coradical filtration} of $C$.
Since $C$ is the union of all its right subcomodules of finite dimension we have that
$$C =\bigcup_n C_n.$$
\ \\
If $S$ is a simple right (left) comodule, we will denote by $E(S)$ its injective envelope. We can assume $E(S) \subseteq C$ for every right simple subcomodule $S$ of $C$ (the same for left simple comodules) and have
$$C =\bigoplus_S E(S),$$
where $S$ ranges among the terms of a decomposition $C_0=\bigoplus\limits_iS_i$ into a direct sum of simple right (left) comodules.

Every right (left) $C$-comodule can be thought as a left (right) $C^*$-module via the action $f \rightharpoonup m = \sum m_0f(m_1)$ (respectively  $m \leftharpoonup f =\sum f(m_{-1})m_0)$. In particular $C$ will be a right and a left $C^*$-module.
Left (right) $C^*$-modules $M$ arrising this way from right (left) $C$-comodules are called \emph{rational} left (right) $C^*$-modules. That is, a left $C^*$-modujle $M$ is rational if for which for every $m\in M$ there are $(m_i)_{i=1,\dots,n}\in M$ and $(c_i)_{i=1,\dots,n}\in C$ such that $f \rightharpoonup m = \sum m_if(c_i)$ for all $f\in C^*$. This means that the category of right (left) $C$-comodules is equivalent to the category of left (right) rational $C^*$-modules.

\subsection*{Profinite and Pseudocompact Algebras}
The following statements are equivalent for an algebra $A$:
\begin{itemize}
\item $A$ is an inverse limit of finite dimensional algebras.
\item $A$ admits a Hausdorf and complete topological algebra structure, with a basis of neighborhoods of $0$ consisting of two sided ideals of finite codimension.
\item $A = C^*$, for some coalgebra $C$.
\end{itemize}
We refer the reader to \cite{DNR} for more details. In the first case, we say that $A$ is a profinite algebra. In the second case, $A$ carries a topology $\tau$ and we say that $(A,\tau)$ is a pseudocompact algebra. This topology comes from the structure of $C$, and $C^*$ is a pseudocompact algebra with the basis of neighborhoods of $0$ given by $X^\perp$, where $X$ ranges over the finite dimensional subcoalgebras of $C$. In fact, the category of pseudocompact algebras with continuous morphisms and that of coalgebras with coalgebra morphisms are in duality. However, note that it could be the case that a profinite algebra $A$ is of the form $A\cong C^*$ and $A\cong D^*$ for two different coalgebras $C$ and $D$, and so might admit two different structures of pseudocompact algebra (i.e. different topologies). However, some properties of such algebras do not depend on the topology (i.e. on the coalgebra giving rise to $A$), while others - such as those involving rational modules - require that the coalgebra $C$ is given (fixed). Many of the results we will prove in this paper turn out to be independent of the specific topological structure of $A$ (that is, of the coalgebra $C$ to which the algebra $A$ is dual), and will therefore be stated in the more general context of profinite algebras.

\section{FGF Conjecture for profinite algebras}

Let $R$ be a ring. We say that $R$ is a left D-ring if every left ideal in $R$ is an annihilator (of some subset of $R$; see \cite{FH} for example). It is easy to note that $R$ is a left $D$-ring if and only if every cyclic left $R$-module embeds in a (product) power of $R$. Indeed, any such embedding of some $R/I$ is given by a morphism of left modules $R\rightarrow R^\alpha$, $r\mapsto (rx_i)_{i\in \alpha}$ with kernel $I=\{r|rx_i=0,\,\forall i\}$, which is equivalent to saying that $I$ is the left annihilator of the set $\{x_i|i\in I\}$. Our first result will show that a profinite algebra which is a left (or right) D-ring must be finite dimensional.

In what follows, $A$ will be an algebra over a field $\Bbbk$ which is pseudocompact (or profinite), so $A=C^*$, where $C$ is some $\Bbbk$-coalgebra.

\begin{proposition}\label{1}
 Every left annihilator in $C^*$ is closed under the finite topology of $C^*$.
\end{proposition}
\begin{proof}
 As $I$ is a left annihilator, there is some $H\subseteq C^*$, such that $I=\{f\in C^* \mid f\cdot h=0, \ \forall h\in H\}$. Now note that $f\cdot h=0$ if and only if $\sum f(c_1)h(c_2)=0$, for all $c\in C$, equivalently, $f\vert_{\{\sum c_1h(c_2)\mid c\in C\}}=0$. Thus, if we denote $X=\sum\limits_{h\in H}{\{\sum c_1h(c_2)\mid c\in C\}}$, we have that $f\in I$ if and only if $f\vert_X=0$, so $I=X^\perp$ and thus $I$ is closed.
\end{proof}


\begin{lemma}\label{obs}
Let $(I_n)_n$ be an ascending chain ($I_n\subseteq I_{n+1}$) of closed ideals of $A=C^*$. Then $\bigcup\limits_{n}I_n$ is closed if and only if the chain terminates ($I_n=I_{n+1}=...$).
\end{lemma}
\begin{proof}
If the sequence terminates then the assertion is trivial. For the converse, assume $I=\sum\limits_nI_n=\bigcup\limits_nI_n$ is closed, so $I=X^\perp$. Let $I_n=X_n^\perp$; then $I_n\subseteq I_{n+1}$ yields $X_n\supseteq X_{n+1}$. Note that $X=(X^\perp)^\perp=I^\perp=(\sum\limits_nI_n)^\perp=\bigcap\limits_nI_n^\perp=\bigcap\limits_n X_n$.
If we assume that the chain $X_1\supseteq X_2\supseteq\dots$ does not terminate, we can choose $x_n\in X_n\setminus X_{n+1}$, and we can easily see that the sum of vector spaces $X+\sum\limits_n\Bbbk x_n$ is direct. This allows us to get an $f\in C^*$ such that $f(x_n)=1$, for all $n$ and $f\vert_X=0$ (completed suitably to a linear function on $C$). But then $f\in X^\perp=I$ and $f\notin X_n^\perp=I_n$ for all $n$, so $f\notin\bigcup\limits_nI_n=I$, a contradiction. Therefore $X_n=X_{n+1}=...$ from some $n$ onward, and this shows that the sequence $I_1\subseteq I_2\subseteq\dots I_n\subseteq\dots$ must terminate as well.
\end{proof}

\begin{remark}
Note that the above lemma also shows a remarkable fact about the closed ideals of a pseudocompact algebra $A=C^*$: any such ideal is either finitely generated, or otherwise, it is not countably generated! Indeed, if $I$ is countably generated then it is the union of an ascending chain of finitely generated ideals, which are then closed (see for example \cite[Lemma 1.1]{I}). Since $I$ is closed, $I$ must be the union of only finitely many of these ideals. So the above lemma restates equivalently that any closed countably generated ideal of $C^*$ is finitely generated.
\end{remark}

Part of the following Lemma is found in \cite[Lemma 3.2]{NT}. Although we only need the equivalence of (i) and (ii), we include (iii) with a new short proof which does not use any general module-theoretic results, and only uses the observation of Lemma \ref{obs}.

\begin{lemma}\label{2}
Let $C$ be a coalgebra. The following assertions are equivalent:\\
(i) Any left ideal of $C^*$ is closed.\\
(ii) $C^*$ is left Noetherian.\\
(iii) $C$ is an artinian left comodule (right $C^*$-module).
\end{lemma}
\begin{proof}
(ii)$\Rightarrow$(i) is easy, since any finitely generated left ideal of $C^*$ is closed (this is a well known fact; one can also see \cite[Lemma 1.1]{I}).\\
(i)$\Rightarrow$(iii) Follows by Lemma \ref{obs}: any ascending chain of left ideals of $C^*$ must terminate since their union is closed. Because every ideal is closed, the latices of the left ideals of $C^*$ and left subcomodules of $C$ are in duality by the $X\rightarrow X^\perp$ correspondence, and so $C$ is an artinian $C$-comodule. \\
(iii)$\Rightarrow$(ii) Let $I$ be a left ideal and assume it is not finitely generated; then there is a sequence of ideals $I_n=\sum\limits_{i=1}^nC^*a_i$ with $a_{n+1}\notin I_n$ for all $n$, so $I_1\subseteq I_2\subsetneq\dots\subsetneq I_n\subsetneq\dots$ is an ascending chain of left ideals. They are closed (since they are finitely generated), so $I_n=X_n^\perp$ with $X_n=I_n^\perp$, and we have a descending chain $X_1\supsetneq X_2\supsetneq\dots\supsetneq X_n\supseteq\dots$ of left subcomodules of $C$ which does not terminate (since if $X_n=X_{n+1}$ then $I_n=X_n^\perp=X_{n+1}^\perp=I_{n+1}$) which is a contradiction.
\end{proof}

In the following proof, we will use known facts on Loewy length of comodules, which are also found in \cite{IC}, Section 2.

\begin{theorem}\label{t1}
If a profinite algebra $A$ is a left D-ring, then it is finite dimensional.
\end{theorem}
\begin{proof}
Let $A=C^*$ for some coalgebra $C$. First note that if $C^*$ is a left D-ring then any left ideal of $C^*$ is closed by Proposition \ref{1} and then $C^*$ is left Noetherian by Lemma \ref{2}. Thus, $C$ is artinian as left $C$-comodule by the duality between the latices of left ideals of $C^*$ and left coideals of $C$ (any left ideal of $C^*$ is closed). This shows that all $C_n/C_{n-1}$ are finite dimensional (such $C$ is said to be of finite type), so each $C_n$ is then finite dimensional. We show that $C_n=C$ for some $n$. Assume not; then for each $n$, $C_n^*$ is a cyclic left $C^*$-module, so there is some embedding $C_n^*\hookrightarrow (C^*)^{k(n)}$. The power here can be assumed finite since $C_n^*$ is finite dimensional.  Let $lw(M)$ denote the Loewy length of a semiartinian module $M$. Since $lw(C_n^*)=lw(C_n)=n$, we have that $(C^*)^{k(n)}$ contains left submodules of Loewy length at least $n$, and this  shows that $L_n((C^*)^{k(n)})\neq L_{n-1}((C^*)^{k(n)})$, and then also $L_n(C^*)\neq L_{n-1}(C^*)$. But this contradicts the fact that $C^*$ is left Noetherian, since the sequence of left submodules $L_0(C^*)\subsetneq L_1(C^*)\subsetneq L_2(C^*)\subsetneq \dots$ of $C^*$ does not terminate.
\end{proof}

Using the results of \cite{Nak}, it is shown in \cite[Corollary 7.2]{FH} that a finite dimensional algebra which is a left (or a right) D-ring is necessarily QF. Thus we have a nice generalization of that statement from finite dimensional algebras to profinite ones:

\begin{corollary}\label{D}
A profinite algebra which is a left (or right) D-ring is necessarily a QF ring. Moreover, it is finite dimensional.
\end{corollary}

We can use this to see that the FGF conjecture of Faith, which asks whether a ring $R$ with the property that every finitely generated $A$-module embeds in a free one is necessarily QF, holds for that class of profinite (and of pseudocompact) algebras.

\begin{corollary}
Let $A$ be a profinite algebra. Then the following are equivalent:\\
(i) every finitely generated left $A$-module embeds in a free one.\\
(ii) every cyclic left $A$-module embeds in a free one.\\
(iii) $A$ is a finite dimensional QF algebra.
\end{corollary}
\begin{proof}
Obviously, (i) implies (ii) and (ii) implies that $A$ is a left D-ring, and this implies (iii) by the previous Corollary \ref{D}. Also, (iii) implies (i) and (iii) implies (ii) for any QF ring.
\end{proof}


\section{Profinite PF algebras}

In this section, we approach another conjecture of Faith for the class of profinite algebras. A conjecture of particular interest in ring theory is the following statement:

(C1) If $R$ is a left (or right) semiartinian ring which is left selfinjective, then $R$ is QF.

In this full generality, it is not known whether it is true or not. Even with more restrictive conditions, the question is still open. The next statement has come to be called {\em Faith's Conjecture} by many authors, and is also an important open question. The reader is referred to \cite{FH} for a comprehensive survey.

(FC) If $R$ is a semiprimary ring (i.e. semilocal with nilpotent radical) which is left selfinjective, then $R$ is QF.

Note that, since any semiprimary ring is semiartinian, (C1) is indeed more general than (FC) (in the sense that (C1) implies (FC)).

We prove here (C1) for profinite algebras. Although this does not require the full extent of this section, it seems worthwhile to characterize those profinite algebras which are QF, or more generally left (right) PF or two sided PF. A ring is left PF if and only if it is an injective cogenerator for its left modules and PF if it is left and right PF. A ring $R$ is called left Kasch if every simple left $R$-module embeds in $R$. It is well known that $R$ is left PF if and only if $R$ is left selfinjective and left Kasch. We are also motivated by establishing the connection between these notions and those of (left or right) QcF introduced for coalgebras, that is, the connection between a coalgebra $C$ being left and/or right QcF and $C^*$ being left and/or right PF. One such connection is known (see \cite[3.3.8 \& 3.3.9]{DNR}): if $C$ is a left QcF coalgebra, then $C^*$ is right selfinjective, and $C^*$ is right selfinjective if and only if $C$ is flat as a left $C^*$-module. We introduce the following definition, which seems natural in the setting of pseudocompact algebras:

\begin{definition}
(i) Let $A$ be a pseudocompact algebra; we say it is left (right) Rat-Kasch if any simple pseudocompact $A$-module embeds in $A$. Equivalently, if $C$ is a coalgebra, we say that $C^*$ is left (right) Rat-Kasch if every simple rational left (right) $C^*$-module embeds in $C^*$. \\
(ii) A pseudocompact algebra $A=C^*$ will be called left (right) Rat-PF if it is left (right) Rat-Kasch and left (right) selfinjective.
\end{definition}

We note that as for PF rings, we have that a pseudocompact algebra $A=C^*$ is left Rat-Kasch and left selfinjective if and only if $C^*$ is left selfinjective and cogenerates all rational $C^*$-modules. Indeed, if $C^*$ is left selfinjective and left Rat-Kasch, since every rational left $C^*$-module $M$ has essential socle $M_0$, we can find an embedding $j:M_0\rightarrow (C^*)^\gamma$ into a power of $C^*$. This extends to a morphism $\overline{j}:M\rightarrow (C^*)^\gamma$, which is injective because $M_0$ is essential in $M$. We can then think of these two conditions \emph{left Rat-Kasch and left selfinjective} as being a suitable analogue notion for that of PF-rings in the context of pseudocompact algebras, which we may call \emph{left Rat-PF}.\\
In the following, we show that this notion of Rat-PF for a pseudocompact algebra $A=C^*$ is left-right symmetric and is equivalent to the coalgebra $C$ being QcF. \\
Obvioulsy, if a pseudocompact algebra is left Kasch, then it is also left Rat-Kasch. We inmediately see that:


\begin{proposition}\label{qcfratkasch}
If $C$ is left QcF, then $C^*$ is left Rat-Kasch.
\end{proposition}
\begin{proof}
If $T$ is a rational simple left $C^*$-module, i.e. a simple right $C$-comodule, $T$ embeds in $C$.
Since $C$ is left QcF, there is an embedding $C\hookrightarrow (C^*)^I$, so $T$ embeds in $(C^*)^I$.
Now one of the projections $(C^*)^I\rightarrow C^*$ will restrict to a nonzero morphism $T\rightarrow C^*$
which must be an embedding since $T$ is simple. 
\end{proof}

Following \cite{H}, we say that a category ${\mathcal A}$ is a {\bf quasi-Frobenius category} if it satisfies the conditions:

(1) ${\mathcal A}$ is an abelian category with enough projectives.

(2) All projective objects in ${\mathcal A}$ are injective.

\begin{proposition}\label{QFcat}
Let $C$ be a coalgebra. Then $C$ is right QcF if and only if $C$ is right semiperfect and ${\bf M}^C$ is a QF-category.
\end{proposition}

\begin{proof} Let $S$ be a simple left $C$-comodule. Since $C$ is right semiperfect, $E(S)$ is finite dimensional. Thus, $E(S)^*$ is a finite dimensional right $C$-comodule, which is projective since $E(S)$ is injective. By hypothesis, $E(S)^*$ must also be injective, and this shows that $E(S)$ is projective. Hence, $^CC=\bigoplus_S E(S)$ is projective, and $C$ is right QcF.

Conversely, if $C$ is right QcF, it is well known that $C$ is right semiperfect. Moreover, as $^CC$ is projective, $E(^CS)$ is projective so $E(^CS)^*$ is injective. Since these comodules form a family of projective generators for ${\bf M}^C$, it follows that any projective right $C$-comodule is injective.

\end{proof}

We give one more definition. Recall that a category ${\mathcal A}$ is said to be a {\bf Frobenius} category if it satisfies the following conditions:\\
(1) ${\mathcal A}$ is an abelian category with enough projectives and with enough injectives.\\
(2) All projective objects in A are injective.\\
(3) All injective objects in A are projective

We now provide the connections between the left and right Rat-PF notions for $C^*$, the QcF notions for $C$ and categorical properties of the category of left (and that of right) $C$-comodules.

\begin{theorem}\label{t2}
Let $C$ be a coalgebra. The following assertions are equivalent:\\
(i) $C^*$ is left Rat-Kasch and left selfinjective, i.e. it is left Rat-PF.\\
(ii) $C$ is left and right quasi-co-Frobenius.\\
(iii) $C^*$ is right Rat-Kasch and right selfinjective, i.e. it is right Rat-PF.\\
(iv) ${\bf M}^C$ and $^C{\bf M}$ are Frobenius categories.
\end{theorem}
\begin{proof}
(ii)$\Rightarrow$(i) follows from Propositions \ref{qcfratkasch} and \ref{QFcat} and the observation that $C$-right QcF implies $C^*$ is left selfinjective.\\
For (i)$\Rightarrow$(ii) note that for each right simple comodule $T$ there is an embedding $T\hookrightarrow C^*$. Let $C\cong\bigoplus\limits_{S\in\Ss}E(S)^{n_S}$ be a decomposition of $C$ into indecomposable injective left comodules (right $C^*$-modules), where $S$ ranges through a system of representatives $\Ss$ for simple left $C$-comodules, and $n_S$ is the multiplicity of each $E(S)$ in $C$. Similarly let $C=\bigoplus\limits_{T\in\Tt}E(T)^{p_T}$ be a decomposition of $C$ into indecomposable injective right comodules (see \cite[Chapter 2]{DNR}). Then $C^*\cong\prod\limits_{S\in\Ss}E(S)^*{}^{n_S}$ as left $C^*$-modules, and the embedding $T\hookrightarrow C^*$ produces an embedding $T\hookrightarrow E(S)^*$ for some $S\in\Ss$ as in the proof of the previous proposition. Since $T$ is rational, $T\subseteq Rat(E(S)^*)$. Now since the functor $Rat:{}_{C^*}\Mm\rightarrow \Mm^C=Rat({}_{C^*}\Mm)$ is a right adjoint to the inclusion functor $i:\Mm^C \hookrightarrow {}_{C^*}\Mm$, and because $i$ is exact, we have that $Rat$ preserves injective objects. This shows that $Rat(E(S)^*)$ is injective in $\Mm^C$, because $E(S)^*$ is injective since it is a direct summand in $C^*$ as left $C^*$-modules. Therefore, it follows that there is an embedding of the injective envelope of $T$ into $Rat(E(S)^*)$, $E(T)\subseteq Rat(E(S)^*)\subseteq C^*$. This gives an embedding of left $C^*$-modules $C=\bigoplus\limits_{T\in\Tt}E(T)^{p_T}\hookrightarrow \coprod\limits_\alpha (C^*) \hookrightarrow (C^*)^\alpha$ (for a suitable set $\alpha$) which shows that $C$ is left QcF. \\
In particular, $C$ is left semiperfect (see \cite[Chapter 3]{DNR}), i.e. $E(T)$ is finite dimensional for all $T\in\Tt$. Note that whenever $Rat(E(S)^*)\neq 0$, we can find an injective indecomposable subobject $E(T)\subseteq E(S)^*$ which would then split off since the finite dimensional injective comodule $E(T)$ is injective also as $C^*$-module (see \cite[Section 2.4]{DNR}). But since $E(S)^*$ is indecomposable (see \cite[Lemma 1.4]{IF}), we get that $E(T)=E(S)^*$. Therefore, for each $S\in \Ss$ there are 2 possibilities: either $E(S)^*$ is finite dimensional and there is some $T\in \Tt$ such that $E(S)^*\cong E(T)$ - denote $\Ss_0$ the set of these $S\in\Ss$ or $Rat(E(S)^*)=0$ - denote $\Ss'=\Ss\setminus \Ss_0$. \\
We now claim that $Rat({}_{C^*}C^*)=\bigoplus\limits_{S\in\Ss_0}(E(S)^*)^{n_S}$. If $r\in Rat({}_{C^*}C^*)=Rat(\prod\limits_{S\in\Ss}(E(S)^*)^{n_S})$, then any projection $\pi_S$ of $C^*$ to an $E(S)^*$ will give an element in $Rat(E(S)^*)$. Therefore, it follows that the coordinates $\pi_S(r)$ of $r$ corresponding to $S\in\Ss'$ are $0$, i.e. $r\in \prod\limits_{S\in\Ss_0}E(S)^*$ so $r\in Rat(\prod\limits_{S\in\Ss_0}E(S)^*)$. Now note that $\Sigma=\bigoplus\limits_{S\in\Ss_0}E(S)^*$ is a quasi-finite right comodule, that is $\Hom^C(T,\Sigma)$ is finite dimensional for all simple $T\in \Tt$. This is because each finite dimensional $E(S)^*$ is obviously isomorphic to a different $E(T)$. Then, by \cite[Example 2.9]{Iprs} (and the proof therein) it follows that $\Sigma=\bigoplus\limits_{S\in\Ss_0}E(S)^*=Rat(\prod\limits_{S\in\Ss_0}E(S)^*)$, which proves the claim. \\
Finally, we note that $\bigoplus\limits_{\NN}C\cong \bigoplus\limits_{\NN}Rat({}_{C^*}C^*)$. Indeed, we have
\begin{eqnarray*}
\bigoplus\limits_{\NN}C & \cong & \bigoplus\limits_{\NN}\bigoplus\limits_{T\in
\Tt}E(T)^{p_T} \cong \bigoplus\limits_{T\in\Tt}\coprod\limits_{\NN}E(T)\\
& \cong & \bigoplus\limits_{S\in\Ss_0}\coprod\limits_{\NN}E(S)^* \cong \bigoplus\limits_{\NN}\bigoplus\limits_{S\in \Ss}(E(S)^*)^{n_S}\\
& \cong &\bigoplus\limits_{\NN}Rat({}_{C^*}C^*)
\end{eqnarray*}
Therefore, $C$ and $Rat({}_{C^*}C^*)$ are weakly $\sigma$-isomorphic, in the terminology of \cite{IG}, i.e. coproduct powers of these objects are isomorphic as left $C^*$-modules: $C^{(\NN)}\cong (Rat({}_{C^*}C^*))^{(\NN)}$. It follows then by \cite[Theorem 1.7]{IG} that $C$ is (left and right) QcF.\\
(ii)$\Leftrightarrow$(iii) follows similarly.\\
From Proposition \ref{QFcat}, we have that (ii) implies that ${\bf M}^C$ and $^C{\bf M}$ are $QF$ categories. It is well known that both categories have enough injectives. Moreover, as $C$ is right QcF, it is projective as a left $C$-comodule, so every injective left comodule is projective and therefore $^C{\bf M}$ is a Frobenius category. Similarly, from $C$ being left QcF we deduce that ${\bf M}^C$ is a Frobenius category. So we get (ii) implies (iv).\\
Conversely, if ${\bf M}^C$ is a Frobenius category, then the (injective object) $C$ is projective as a left $C$-comodule and therefore $C$ is right QcF. Similarly, from $^C{\bf M}$ being a Frobenius category, we deduce that $C$ is left QcF. So we have (iv) implies (ii) and we are done.
\end{proof}

We can summarize the results on PF-related properties of profinite algebras:

\begin{theorem}
Let $A$ be a profinite algebra. Then the following are equivalent:\\
(i) $A$ is a left PF ring.\\
(ii) $A$ is a right PF ring.\\
(iii) $A$ is left cogenerator.\\
(iv) $A$ is right cogenerator.\\
(v) $A$ is a QF ring.\\
(vi) $A$ is a finite dimensional QF algebra.\\
(vii) $A$ is a left FGF ring.\\
(viii) $A$ is a right FGF ring.\\
(ix) $A$ is left CF.\\
(x) $A$ is right CF.\\
(xi) $A$ is a left D-ring.\\
(xii) $A$ is a right D-ring.
\end{theorem}
\begin{proof}
We have (v)$\Rightarrow$(i)$\Rightarrow$(iii)$\Rightarrow$(vii)$\Rightarrow$(ix)$\Rightarrow$(xi) and (xi)$\Rightarrow$(v) from Corollary \ref{D}. Similarly, the equivalences to the right follow.

\end{proof}

\begin{remark} Recall that a ring $R$ is left FPF (finitely pseudo-Frobenius) if every finitely generated faithful left module is a generator. If $A$ is a left semiartinian and left FPF profinite algebra, then $A$ is semilocal (see Proposition \ref{semi}) and by Tachikawa's theorem \cite[Theorem 1.9]{FP} A is left PF. Hence by the preceeding result A is QF. This answers questions 9 and 10 of Faith's book \cite{FP} for profinite algebras.
\end{remark}

\section{(C1) and (FC) for profinite algebras}

For each type $S$ of simple left $C$-comodule denote $C_S$ its associated coalgebra, that is, $C_S=\sum\{S'|S\cong S'\subset C\}$, and let $A_S=C_S^*$. We then have $C_0=\sum\limits_{S\in\Ss}C_S$, and $Jac(C^*)=C_0^\perp$, $C^*/Jac(C^*)=C^*/C_0^\perp\cong C_0^*\cong \prod\limits_{S\in\Ss}A_S$. Here $C_S$ are simple coalgebras and $A_S$ are simple finite dimensional algebras.
 We have then that a pseudocompact algebra $C^*$ is semilocal if and only if $C_0$ is finite dimensional. Indeed, if $C_0$ is finite dimensional (i.e. there are only finitely many isomorphism types of simple left (right) comodules), as $A/Jac(A)\cong C_0^*$, we get that $C_0^*$ is finite dimensional semisimple and there are only finitely many types of simple left (right) $A$-modules. The converse statement follows easily too. Moreover, in this case, every simple left (right) $A$-module, being a left (right) $\frac{A}{J(A)}$ module, is rational.

\begin{proposition}\label{semi}
Let $C$ be a coalgebra such that $C^*$ is left (or right) semiartinian. Then $C^*$ is semilocal.
\end{proposition}

\begin{proof}
It follows from \cite[Proposition 1.1]{INT}.
\end{proof}

\begin{theorem}\label{FC} [(C1) for Profinite Algebras]
If $A$ is a profinite algebra which is semiartinian and left self-injective then $A$ is a finite dimensional QF algebra.
\end{theorem}
\begin{proof}
Let $A=C^*$, for $C$ a coalgebra. With notations as before, as $A$ is semiartinian, for each simple left comodule $S$, $E(S)^*$ contains some simple left module $T$. As $A$ is semilocal, $T$ is rational. Again as before, we note that in fact $E(T)$ must embed in $E(S)^*$, so let $\varphi:E(T)\rightarrow E(S)^*$ be a morphism of left $C^*$-modules. This yields a morphism $\psi:E(S)\rightarrow E(T)^*$, given by $\psi(y)(x)=\varphi(x)(y)$. We see that $\psi(E(S))\not\subset T^\perp$, where $T^\perp=\{y^*\in E(T)^*| y^*(T)=0\}$. Indeed, otherwise $\varphi(T)(E(S))=\psi(E(S))(T)=0$, so $\varphi(T)=0$ i.e. $T\subseteq \ker(\varphi)=0$ which is not true. Again using \cite[Lemma 1.4]{IF} we have that $E(T)^*$ is cyclic local with $T^\perp$ its unique maximal, and since $\psi(E(S))\not\subset T^\perp$ we get $\psi(E(S))=E(T)^*$. Therefore, $E(T)^*$ is rational as a quotient of a rational module ($E(S)$) and so it is finite dimensional since it is cyclic. Proceeding as before, we get that $E(T)$ is injective and it splits off in $E(S)^*$, and therefore we must have $E(T)\cong E(S)^*$. This reasoning holds for each simple left comodule $S$, so $C$ is finite dimensional, since there are only finitely many simple left $C$-comodules. Therefore $A$ is a finite dimensional left selfinjective algebra, thus QF.
\end{proof}


\begin{corollary} [Faith's Conjecture (FC) for Profinite Algebras]
If $A$ is a profinite semiprimary algebra which is left selfinjective, then $A$ is a finite dimensional QF algebra.
\end{corollary}
\begin{proof}
Immediate from Theorem \ref{FC}.
\end{proof}

\section{Some more remarks}
We note an interesting fact, which shows that the left Kasch profinite algebras are a priori more complicated than the left Rat-Kasch profinite algebras:

\begin{lemma}\label{nonratsimples}
Let $A$ be a profinite algebra and assume there are infinitely many types of simple right (or, equivalently, of left) rational $A$-modules. Then there are simple non-rational left (and also right) $A$-modules.
\end{lemma}
\begin{proof}
It is enough to prove this for the cosemisimple case, since any simple $A$-module is a simple $A/Jac(A)\cong C_0^*$-module and viceversa. It is a fairly well known fact that the maximal two-sided ideals in $C_0^*=\prod\limits_{S\in\Ss}A_S$ are in one to one correspondence with ultrafiters $\Uu$ on the power set of $\Ss$ ($A_S$ are simple algebras), and that this correspondence takes principal ultrafilters into maximal two-sided ideals $M$ of the form $M_{S_0}=\prod\limits_{S\in\Ss\setminus \{S_0\}}A_S$ for $S_0\in\Ss$, and $A/M_{S_0}$ is a simple finite dimensional algebra whose only type of simple left module is (isomorphic to) $S_0^*$. In fact, $M_{S_0}=ann(S_0^*)$, for all $S_0\in\Ss$. \\
It is then enough to take a non-principal ultrafilter $\Uu$ (e.g. one containing the Frechet fitler) and consider a simple left $\frac{A}{M}$-module $T$ (they are all isomorphic). Then $T\not\cong S_0^*$ for $S_0\in \Ss$, since otherwise $M\subseteq ann(T)=ann(S_0^*)=M_{S_0}$, so $M=M_{S_0}$, which is impossiple by choice. But as $S_0^*$ for $S_0\in\Ss$ are the simple rational left $C^*$-modules (simple right $C$-comodules), it follows that $T$ is not rational.
\end{proof}

We now characterize the left Kasch property for profinite algebras.

\begin{proposition}\label{fdsimple}
Let $A$ be a profinite algebra, and let $C$ be any coalgebra such that $A=C^*$. If $H$ is a simple $A$-module which embeds in $A$, then $H$ is a rational $C^*$-module.
\end{proposition}
\begin{proof}
Assume there is $H\hookrightarrow C^*$ be a monomorphism of left $C^*$-modules. Since $C^*=\prod\limits_{C_0=\bigoplus S} E(S)^*$ is a product of duals of injective envelopes of simple left $C$-comodules, we can find a nonzero projection $\psi:H\rightarrow E(S)^*$, which is injective since $H$ is simple. Fix $h\in H, h\neq 0$. By the injectivity, $\psi(h)\neq 0$ so there is $x\in E(S)$ such that $\psi(h)(x)\neq 0$. Let $N=x\cdot C^*$, which is a finite dimensional left $C$-comodule. Let $\psi_0:H\rightarrow N^*$ defined by $\psi_0(s)=\psi(s)\vert_{N}$. Note that for any $0\neq s\in H$, there is $c^*\in C^*$ such that $h=c^*\cdot s$, because $H=C^*\cdot s$ ($H$ is simple). Then we have
\begin{eqnarray*}
\psi_0(s)(x\cdot c^*) & = & \psi(s)(x\cdot c^*) = \psi(c^*\cdot s)(x) \\
& = & \psi(h)(x) \neq 0
\end{eqnarray*}
This shows that $\psi_0(s)\neq 0$. Therefore, $\psi_0$ is injective. Thus, $H$ embeds in $N^*$, and since $N^*$ is rational (since $N$ is a finite dimensional left $C$-comodule), it follows that $H$ is a simple rational $C^*$-module.
\end{proof}

\begin{corollary}
The following are equivalent for a profinite algebra $A$.\\
(i) $A$ is left Kasch.\\
(ii) $A$ is semilocal and left Rat-Kasch.\\
(iii )$A$ is almost connected (i.e. $A/Jac(A)$ is finite dimensional) and left Rat-Kasch.\\
(Note: if $A$ is semilocal, then left (right) Rat-Kasch does not depend on the coalgebra $C$ for which $C^*=A$).
\end{corollary}
\begin{proof}
Let $C$ be a coalgebra such that $C^*=A$. If $C_0$ is infinite dimensional, by Proposition \ref{nonratsimples} it follows that there are some non rational $C^*$-modules. But by Proposition \ref{fdsimple} if $A$ is left (or right) Kasch, then any simple left $A$-module is $C^*$-rational. This shows (i)$\Rightarrow$(ii),(iii). For the converse, note that if $A=C^*$ is semilocal, equivalently, $C_0$ is finite dimensional, then any simple $A=C^*$-module is $C^*$-rational, since $A/Jac(A)\cong C_0^*$ is finite dimensional semisimple. Thus, in this case, left Rat-Kasch is equivalent to left Kasch.
\end{proof}


\end{document}